\documentclass[12pt]{amsart}
\usepackage{latexsym} 
\usepackage{amsmath,amssymb,amsfonts,amsthm}
\usepackage{verbatim}
\usepackage[dvips]{graphicx}
\newtheorem{theorem}{Theorem}

\newtheorem{remark}{Remark}
\newtheorem{proposition}{Proposition}
\newtheorem{corollary}{Corollary}
\newtheorem{lemma}{Lemma}
\newtheorem{example}{Example}

\theoremstyle{definition}
\newtheorem{definition}{Definition}

\begin{document}

\noindent

  \title{ On Stochastic generalized functions}

\thanks{Research partially supported by FAPESP 02/10246-2, 2007/54740-4 and CNPQ 302704/2008-6.}

\begin{center}
\end{center}

\author{Pedro Catuogno  and Christian Olivera }
\address{Departamento de
 Matem\'{a}tica, Universidade Estadual de Campinas,\\ 13.081-970 -
 Campinas - SP, Brazil.}
\email{pedrojc@ime.unicamp.br ; colivera@ime.unicamp.br}

\begin{abstract}\vskip.2in
We introduced a new algebra of stochastic generalized functions
which contains to the space of stochastic distributions
$\mathcal{G}^{\ast}$, \cite{PoTi}. As an application, we prove
existence and uniqueness of the solution of a stochastic Cauchy
problem involving singularities.
\end{abstract}

\maketitle

\section {Introduction}

The algebras of generalized functions were introduced by J. F.
Colombeau \cite{col1} and has been studied and developed by many
authors (see \cite{Biag}, \cite{col1}, \cite{gkos} and
\cite{ober2} and references). These algebras of generalized
functions contain the Schwartz distributions, thus one can deal
with its multiplication and other types of nonlinearities. Many
applications have been carried out in various fields of
mathematics such as partial differential equations, Lie analysis,
local and microlocal analysis, probability theory, differential
geometry. In particular, the algebras of generalized functions
found applications in non-linear partial differential equations
(\cite{Biag}, \cite{col1}, \cite{gkos} and \cite{ober2}), where
the classical distributional methods are limited by the
impossibility to consistently define an intrinsic product of
distributions \cite{Schw}.

\noindent In recent years stochastic partial differential equations
have been studied actively. Often, solutions to such equations do
not exist in the usual sense, rather they are generalized functions.
It is possible consider the solutions as generalized functions in
the space-time variables or as generalized stochastic functions
which are point-wise defined in space and time. In this work we take
the latter approach, more precisely we consider stochastic process
in algebras of generalized functions, following to F. Russo and
their coauthors (see for instance \cite{albe}, \cite{albe2}
,\cite{ober4} and \cite{Russo}.

\noindent We considerer the following stochastic Cauchy problem,
\begin{equation}\label{intro}
 \left \{
\begin{array}{lll}
u_{t}& = & Lu  + u   W(t,x) \\
u(0,x) & = & f(x)
\end{array}
\right .
\end{equation}
where the product is taken in  the usual form,  $L$ is an
uniformly elliptic partial differential operator, $f\in
C_{b}^{2}(\mathbb{R}^{l})$ and $W$ is a space-time white noise
\cite{Hida}. There are a great interest into solve the problem
(\ref{intro}) (see for instance \cite{Chow}, \cite{HOBZ},
\cite{Kuni}, \cite{LORo}, \cite{liouz}, \cite{NuZa}, \cite{PVW}
and \cite{Wals}). We observe that the properties of the solutions
(It\^o-Stratonovich) of (\ref{intro}) depend crucially on the
noise type and its dimension (see \cite{liouz} and \cite{Wals}).

\noindent The aim of this article is solve the stochastic Cauchy
problem (\ref{intro}) without restriction on the dimension of the
white noise $W$. Our approach transfer the generalized nature of
the solutions to the stochastic component (see \cite{HOBZ}  for
more details). It is thus necessary to introduce random variables
with values in algebras of generalized functions. The fundamental
idea is to introduce a new algebra of generalized functions of
Colombeau type in the context of the Hida stochastic
distributions. This algebra of generalized functions contains the
stochastic distributions $\mathcal{G}^{\ast}$ introduced by J.
Potthoff and M. Timpel in \cite{PoTi}. In this context and under
general assumptions we show existence and uniqueness for the
stochastic Cauchy problem (\ref{intro}). Furthermore, others
aspects of these generalized functions are studied.

\noindent The plan of exposition is as follows: in section 2, we
provide some basic concepts on  classical white noise theory,
these are, Wiener-It\^o chaos expansions, stochastic distributions
and Wick products. In section 3, we introduce three definitions of
products of stochastic distributions and we study some properties.
In section 4, we give a new algebra of stochastic generalized
functions $\mathcal{G}$, this algebra contains to the stochastic
distributions $\mathcal{G}^{\ast}$ ( see \cite{PoTi}) and extends
the product in the test space $\mathcal{G}$. Moreover, we study
its elementary properties and we show  that the symmetric product
of stochastic distributions , defined in section 3, is associated
with the product in $\mathcal{G}$. Finally, in section 5, we study
a stochastic Cauchy problem with space-time depend noise white in
the context of these generalized functions.

\section{White Noise Theory }

\subsection{Preliminaries }

\noindent Let $\mathcal{S}(\mathbb{R}^{d})$ be the space of
infinitely differentiable functions from $\mathbb{R}^d$ to
$\mathbb{R}$ $(\mathbb{C})$ which together with all its
derivatives are of rapidly decreasing. We consider
$\mathcal{S}(\mathbb{R}^{d})$ with the topology given by the
family of seminorms
\[
\| f\|_{\alpha,\beta}=\sup\{|x^{\alpha}D^{\beta}f(x)| : x \in
\mathbb{R}^{d}\}
\]
where $\alpha, \beta \in \mathbb{N}_{0}^{d}$. It follows that
$\mathcal{S}(\mathbb{R}^{d})$ is a nuclear space. The Schwartz
space of tempered distributions is the dual space
$\mathcal{S}^{\prime}(\mathbb{R}^{d})$.

\noindent  We shall summarize the definitions and basic properties
of Hermite polynomials and Hermite functions. The $j$-th
{\it{Hermite polynomial}}, denoted by $h_j$, is defined to be

\begin{equation}\label{forpolHe}
    h_j(x)=(-1)^n e^{\frac{x^2}{2}} \frac{d^j}{dx^j}  e^{-\frac{x^2}{2}}
\end{equation}
for $j\in \mathbb{N}_{0}$.

\noindent We observe that
\begin{equation}\label{forpolHp}
      h_j(x) =2^{-\frac{j}{2}}\sum_{k=0}^{[j/2]}\frac{(-1)^k
      n!( \sqrt{2}x)^{j-2k}}{k!(j-2k)!}.
\end{equation}
The $j$-th {\it{Hermite function}} $\eta_j$  is given by
\begin{equation}\label{relfunpolh}
    \eta_{j}(x)=  (\sqrt{2\pi}n!)^{-\frac{1}{2}}
    e^{{-\frac{1}{4}x^2}}h_j(x)
\end{equation}
for $j\in \mathbb{N}_{0}$.

The following properties of the Hermite functions will be
extremely useful.
\begin{itemize}
\item $\{\eta_{j}: j \in \mathbb{N}_{0} \}$ is an orthonormal
basis of $L^{2}(\mathbb{R})$,

\item $\sup \{| \eta_j(x)|:x \in \mathbb{R}
\}=O(j^{-\frac{1}{4}})$,

\item The $\eta_j$ are eigenfunctions of the number operator
$N+1=\frac{1}{2}(-\frac{d^2}{dx^2}+x^2+1)$,
\[
(N+1)\eta_{j}=(j+1)\eta_{j}
\]
for all $j \in \mathbb{N}_{0}$.
\end{itemize}

\noindent The $\alpha$-th {\it{Hermite function}}
$\eta_{\alpha}:\mathbb{R}^d \rightarrow \mathbb{R}$ ($\alpha \in
\mathbb{N}_{0}^{d}$) is given by
\[
\eta_{\alpha}(x)= \prod_{i=1}^d \eta_{\alpha_{i}}(x_i).
\]
It follows that $\{ \eta_{\alpha} : \alpha \in \mathbb{N}_{0}^{d}
 \}$ is an orthonormal basis of $L^{2}(\mathbb{R}^{d})$.

\noindent Let us denote by $(N+1)_d=(N+1)^{\otimes d}$, the
differential operator
\[
(N+1)^{\otimes d}= \frac{1}{2}(-\frac{\partial^2}{\partial
x_d^2}+x_d^2+1)\cdot \cdot \cdot
\frac{1}{2}(-\frac{\partial^2}{\partial x_1^2}+x_1^2+1).
\]
It is easy to check that for $\alpha, \beta \in \mathbb{N}^d_0$,
\[
(N+1)_d^{\beta}(\eta_{\alpha})=(\alpha +1)^{\beta}\eta_{\alpha}
\]
where
$(\alpha+1)^{\beta}=\Pi_{i=1}^{d}(1+\alpha_{i})^{\beta_{i}}$.

\noindent The topology of $\mathcal{S}(\mathbb{R}^{d})$ has an
alternative description in terms of the family of norms $ \{ |
\cdot |_{\beta} : \beta\in \mathbb{N}_{0}^{d}\}$. The norm $|
\cdot |_{\beta}$ is given by
\[
|\varphi|^2_{\beta}:=|(N+1)_d^{\beta}(\varphi)|^2_{0}=\sum_{\alpha\in
\mathbb{N}^d_{0}}(\alpha+1)^{2\beta}<\varphi,\eta_{\alpha}>^2,
\]
where $<\varphi,\eta_{\alpha}>=\int \varphi(x)\eta_{\alpha}(x)dx$
are the Fourier-Hermite coefficients of the expansion of
$\varphi$.

\noindent It follows easily that the seminorm families $ \{ |
\cdot |_{\beta} : \beta\in \mathbb{N}_{0}^{d}\}$ and $ \{\| \cdot
\|_{\alpha,\beta} : \alpha,\beta\in \mathbb{N}_{0}^{d} \}$ are
equivalent on $\mathcal{S}(\mathbb{R}^{d})$.

\noindent The Hermite representation theorem for
$\mathcal{S}(\mathbb{R}^{d})$
($\mathcal{S}^{\prime}(\mathbb{R}^{d})$) states a topological
isomorphism from $\mathcal{S}(\mathbb{R}^{d})$
($\mathcal{S}^{\prime}(\mathbb{R}^{d})$) onto the space of
sequences $\mathbf{s}_{d}$ ($\mathbf{s}_{d}^{\prime}$).

\noindent Let $\mathbf{s}_{d}$ be the space of sequences
\[
\mathbf{s}_{d}=\{(a_\beta)\in
\ell^2(\mathbb{N}^{d})):\sum_{\beta}(\beta+1)^{2\alpha}\mid
a_{\beta}\mid^2<\infty, \; \mbox{for all } \; \alpha \in
\mathbb{N}_{0}^{d}\}.
\]

\noindent The space $\mathbf{s}$ is a locally convex space with the
sequence of norms
\[
| (a_\beta)|_\alpha = (\sum_{\beta}(\beta+1)^{2\alpha}\mid
a_{\alpha}\mid^2)^{\frac{1}{2}}
\]
\noindent The topological dual space to $\mathbf{s}_{d}$, denoted by
$\mathbf{s}_{d}^{\prime}$, is given by
\[
\mathbf{s}_{d}^{\prime}=\{(b_{\beta}): \mbox{for some
}\;(C,\alpha)\in\mathbb{R}\times \mathbb{N}_{0}^{d}, \; \mid
b_{\beta} \mid \leq C(\beta+1)^{\alpha} \mbox{ for all } \beta \},
\]
\noindent and the natural pairing of elements from $\mathbf{s}_{d}$
and $\mathbf{s}_{d}^{\prime}$, denoted by $\langle \cdot, \cdot
\rangle$, is given by
\[
\langle (b_{\beta}), (a_{\beta}) \rangle = \sum_{\beta}
b_{\beta}a_{\beta}
\]
for $(b_{\beta}) \in \mathbf{s}_{d}^{\prime}$ and $(a_{\beta}) \in
\mathbf{s}_{d}$.

\noindent It is clear that $\mathbf{s}_{d}^{\prime}$ is an algebra
with the pointwise operations:
\begin{eqnarray*}
(b_{\beta})+(b^{\prime}_{\beta})& = & (b_{\beta}+b^{\prime}_{\beta}) \\
(b_{\beta})\cdot(b^{\prime}_{\beta})& = &
(b_{\beta}b^{\prime}_{\beta}),
\end{eqnarray*}
and $\mathbf{s}_{d}$ is an ideal of $\mathbf{s}_{d}^{\prime}$.

\noindent The relation between $\mathbf{s}_{d}$
($\mathbf{s}_{d}^{\prime}$) and $\mathcal{S}(\mathbb{R}^{d})$
($\mathcal{S}^{\prime}(\mathbb{R}^{d})$) is induced by the Hermite
functions, via the Fourier-Hermite coefficients.
\begin{theorem}\label{rpreS}
N-representation theorem for $\mathcal{S}(\mathbb{R}^{d})$ and
$\mathcal{S}^{\prime}(\mathbb{R}^{d})$:

$\mathbf{a)}$ Let
$\mathbf{h}:\mathcal{S}(\mathbb{R}^{d})\rightarrow \mathbf{s}_{d}$
be the application
\[
\mathbf{h}(\varphi)=(<\varphi,\eta_{\beta}>).
\]
Then $\mathbf{h}$ is a topological isomorphism. Moreover,
\[
\|\mathbf{h}(\varphi)\|_\beta=\|\varphi\|_\beta
\]
for all $\varphi \in \mathcal{S}(\mathbb{R}^{d})$.

\noindent $\mathbf{b)}$ Let
$\mathbf{H}:\mathcal{S}^{\prime}(\mathbb{R}^{d})\rightarrow
\mathbf{s}_{d}^{\prime}$ be the application
$\mathbf{H}(T)=(T(\eta_{\beta}))$. Then $\mathbf{H}$ is a
topological isomorphism. Moreover, if $T \in
\mathcal{S}^{\prime}(\mathbb{R}^{d})$ we have that
\[
T=\sum _{\beta} T(\eta_{\beta})\eta_{\beta}
\]
in the weak sense and for all $\varphi \in
\mathcal{S}(\mathbb{R}^{d})$,
\[
T(\varphi)=\langle \mathbf{H}(T),\mathbf{h}(\varphi)\rangle.
\]
\end{theorem}
\begin{proof}
See for instance, M. Reed and B. Simon \cite{Red1} pp. 143.
\end{proof}
\noindent We say that the sequences $\mathbf{h}(\varphi)$ and
$\mathbf{H}(T)$ are the {\it Hermite coefficients} of the tempered
function $\varphi$ and the tempered distribution $T$, respectively. \\

\subsection{White noise and chaos expansions}
\noindent In this subsection we briefly recall some of the basic
concepts and results of White noise analysis. Our presentation and
notation will follow \cite{Hida} and \cite{HOBZ}. According to the
Bochner-Minlos theorem, there exits an unique probability $\mu$ on
$\mathcal{B}_d$ the Borel $\sigma$-field of
$\mathcal{S}^{\prime}(\mathbb{R}^{d})$ such that for each $\phi
\in \mathcal{S}(\mathbb{R}^{d})$,
\begin{equation}\label{medi}
\int_{\mathcal{S}^{\prime}(\mathbb{R}^{d})} e^{i(\omega, \phi)} \
d\mu(\omega)=e^{-\frac{1}{2}|\phi|_{L^{2}(\mathbb{R}^{d})}^{2}}.
\end{equation}
The probability space $(\mathcal{S}^{\prime}(\mathbb{R}^{d}),
\mathcal{B}_d, \mu )$ is then called the {\it{White noise
probability space}}. We will denote by $(L^{2})$ the space $L^{2}(
\mathcal{S}^{\prime}(\mathbb{R}^{d}), \mathcal{B}, \mu)$.

\noindent  Let $\{\alpha^j : j \in \mathbb{N}_0 \}$ be an
enumeration of the set of multi-indices $\{ \alpha: \alpha \in
\mathbb{N}_{0}^{d}\}$. We can assume that this enumeration has the
following property: If $k<l$ then $\alpha_1^k+...+\alpha_d^k \leq
\alpha_1^l+...+\alpha_d^l$ (see \cite{HOBZ} pag. 19  for details).
It is clear that $\{ \eta_{\alpha}: \alpha \in \mathbb{N}_0^d \}$
can be enumerated by $\{ \eta_{j}: j\in \mathbb{N}_{0}\}$. \noindent
We denote by $ \mathcal{J}$ the set of sequences of nonnegative
integers with compact support.

\noindent The $\alpha$-th {\it{Hermite polynomial}}
$H_{\alpha}:\mathcal{S}^{\prime}(\mathbb{R}^d) \rightarrow
\mathbb{R}$ ($\alpha \in \mathcal{J}$) is given by
\[
H_{\alpha}(\omega)= \prod_{j=1}^{\infty}
h_{\alpha_j}(<\omega,\eta_{j}>).
\]
It follows that $\{ H_{\alpha} : \alpha \in \mathcal{J}
 \}$ is an orthogonal basis of $(L^{2})$ and $\|H_{\alpha} \|^{2}=\alpha
 !$.
\noindent The Wiener-It\^o chaos expansion theorem says that if
$f\in (L^{2})$, then there exists an unique sequence of functions
$f_{n}\in L^{2}(\mathbb{R}^{d})^{\hat{\otimes}n}$ such that
\[
f(\omega)=\sum_{n=0}^{\infty} I_{n}(f_{n})(\omega)
\]
where $I_{n}: L^{2}(\mathbb{R}^{d})^{\hat{\otimes}n} \rightarrow
(L^2)$ is the multiple Wiener integral of order $n$.

\noindent Moreover, we have the isometry

\[
\mathbb{E}[f^{2}]=\sum_{n=0}^{\infty} n! | f_{n} |_0^{2}.
\]
where we denote by $|\cdot |_0$ the norm $|\cdot
|_{L^{2}(\mathbb{R}^{d})^{\hat{\otimes}n}}$.

 \noindent Let $f\in \mathcal{S}(\mathbb{R})$, from (\ref{medi}) for $d=1$ we have that
 $<\cdot,f>$ is a Gaussian random variable with expectation 0 and
 variance $| f |_0$. We observe that for $t\in
 \mathbb{R}$,
\begin{equation}\label{Brow}
B(t):=<\cdot, 1_{[0,t]}>
\end{equation}
is a Brownian motion. Moreover, if $f\in L^{2}(\mathbb{R})$ we
have that
\[
 <\cdot, f>=I_1(f).
\]

\subsection{Stochastic distributions}

\noindent The space of stochastic distributions (also called of
space of regular Hida distributions) $\mathcal{G}^{\ast}$ has been
introduced by J. Potthoff and M. Timpel in \cite{PoTi}. Their
development was motivated by stochastic differential equations of
Wick type and its applications, see for instance \cite{KOPU},
 \cite{BePo}, \cite{BeThe} and \cite{GKS}.

\noindent Let $N$ be the Ornstein-Uhlenbeck operator on $(L^2)$,
we recall that
\[
N(I_n(f))=nI_n(f)
\]
for all $n \in \mathbb{N}_{0}$ and $f \in
L^{2}(\mathbb{R}^{d})^{\hat{\otimes}n}$.

\noindent Let us denote by $\mathcal{G_{p}}$ ($p \in
\mathbb{N}_0$) the $(L^{2})$-domain of $exp(p N)$. It is clear
that $\mathcal{G_{p}}$ with the norm
\[
\| F \|_{p}^{2}:= \| exp(p N)F \|_{0}^{2}=\sum_{n=0}^{\infty}e^{2p
n}n!|F_n|_0^2
\] is a
Hilbert space.

\noindent Let $\mathcal{G}$ be the projective limit of the family
$\{ \mathcal{G_{p}}, p\geq 0 \}$ and $\mathcal{G}^{\ast}$ be its
dual. It follows that $\mathcal{G}^{\ast}$ is the inductive limit
of the family $\{ \mathcal{G_{-p}}, p\geq 0 \}$, where
$\mathcal{G_{-p}}$ is the dual of $\mathcal{G_{p}}$. In this way
we obtained the Gelfand triple
\[
 \mathcal{G} \subset (L^{2})  \subset  \mathcal{G}^{\ast}.
\]
We observe that the duality in this Gelfand triple is given by
\[
\ll  F, f \gg=\sum_{n=0}^{\infty} n! \ <F_{n},f_{n}>
\]
where $f=\sum_{n=0}^{\infty} I_{n}(f_{n})  \in \mathcal{G}$ and
$F= \sum_{n=0}^{\infty} I_{n}(F_{n}) \in \mathcal{G}^{\ast}$.
\begin{example}
The Brownian motion $B(t)=<\cdot, 1_{[0,t]}>$ is an element of
$\mathcal{G}$. J. Potthof and M. Timpel \cite{PoTi} showed that
for $x \in \mathbb{R}$ the Donsker delta function $\delta_x \circ
B(t)$ belong to $\mathcal{G}^{\ast}$.
\end{example}

\noindent Let $F \in \mathcal{G}^{\ast}$ with Wiener-It\^o
expansion $\sum_{i=0}^{\infty}I_n(F_n)$. The {\it$S$-transform} of
$F \in \mathcal{G}^{\ast}$, denoted by $S(F)$, is the function
from $\mathcal{S}(\mathbb{R}^d)$ to $\mathbb{C}$ given by
\[
S(F)(\varphi)=\sum_{n=0}^{\infty}<F_n,\varphi^{\otimes n}>.
\]

\noindent The {\it Wick product} $ F \diamond  G $ of $F$ and $G$
in $\mathcal{G}^{*}$, is given by $S^{-1}(S(F)S(G)) $. This
product is well defined and continuous from $\mathcal{G}^{*}
\times \mathcal{G}^{*}$ to $\mathcal{G}^{*}$, see \cite{BeThe},
\cite{HuYa} and \cite{PoTi}.

\noindent We observe that the Wick product is associative,
commutative and $\mathcal{G}$ and $\mathcal{G}^{\ast}$  are
algebras respect to the Wick product. An easy computation shows
that if $F=\sum_{i=0}^{\infty}I_i(F_i)$ and
$G=\sum_{i=0}^{\infty}I_i(G_i)$ belong to $\mathcal{G}^{*}$ then
the Wiener-It\^o coefficients of the Wick product $ F \diamond
G=\sum_{i=0}^{\infty}I_i(H_i) $ verify that
\[
H_{i}= \sum_{i=m+n} \ F_{n} \hat{\otimes} G_{m}.
\]

\noindent The {\it Wick exponential of} $f \in L^2(\mathbb{R}^d)$,
denoted by $e^{: <\cdot, f>}$, is the element of $(L^2)$ defined
by
\begin{equation}\label{exwick}
\sum_{m=0}^{\infty}\frac{1}{m!}I_m(f^{\otimes^m})=e^{ <\cdot, f>-
\frac{1}{2}|f|_0^2}.
\end{equation}
\noindent We observe the following useful identity
\begin{equation}\label{exwick1}
\| e^{: <\cdot, f>}
\|^2_p=\sum_{m=0}^{\infty}\frac{e^{2pm}}{m!}|f|_0^{2m}=e^{e^{2p}|f|_0^2}.
\end{equation}

\noindent The space $\mathcal{G}$ is an algebra for the pointwise
multiplication and this multiplication is continuous from
$\mathcal{G}\times \mathcal{G}$ into $\mathcal{G}$, see Theorem 2.5
of \cite{PoTi} ($d=1$) and \cite{BeThe} for the general case. In
particular, for all $m \in \mathbb{N}_0$ there exists $r, s \in
\mathbb{N}_0$ and a constant $C_m>0$ such that
\begin{equation}\label{eq1}
\|\varphi \psi \|_m \leq C_m \| \varphi \|_r \| \psi \|_s
\end{equation}
for all $\varphi, \psi \in \mathcal{G}$.

\noindent Thus  we can define the product of a distribution by a
test function.

\begin{definition} Let $\varphi \in \mathcal{G}$ and $F\in \mathcal{G}^{*}$. We define $\varphi F \in \mathcal{G}^{*}$
by
\[
\ll \varphi F, \psi \gg=\ll F, \varphi \psi \gg
\]
for all $\psi \in \mathcal{G}$.
\end{definition}
\noindent It is clear from (\ref{eq1}) that  $\varphi F$ is well
defined.

\noindent Let $\varphi=\sum_{n=0}^{\infty} I_{n}(\varphi_{n})\in
\mathcal{G}$ and $F=\sum_{m=0}^{\infty} I_{m}(F_{m}) \in
\mathcal{G}^{\ast}$. Then $\varphi F=\sum_{j=0}^{\infty}
I_{j}(H_{j})$ where

\begin{equation}\label{ten2}
H_{j}=\sum_{j=m+n} \sum_{k=0}^{m \wedge n} k! \  {n+k \choose k} \
{m+k \choose k} \ \varphi_{n+k} \hat{\otimes}_{k} F_{m+k}.
\end{equation}
Here $\varphi_{n} \hat{\otimes}_{k} F_{m} $ is the symmetrization
of $\varphi_{n} \otimes_{k} F_{m}$. See for instance \cite{HuYa}
pag. 70.

\section{Distributional products via approximation in chaos}

In this  section, strongly inspired in \cite{calo}, we define
three new products of stochastic distributions. A possible
approach to define a product of a pair of distributions is
approximate one of them by test functions, multiply this
approximation by the other distribution, and pass to a limit. In
the case of the sequential approach (see \cite{Miku} pp 242,
\cite{ober2}) the approximation is done by convolution with
$\delta$-sequences, in this work we propose take the approximation
given by the Wiener-It\^o chaos expansion of the distribution.

\noindent In order to define the chaos approximation we consider
the projections $\Pi_{m}: \mathcal{G}^{\ast} \rightarrow
\mathcal{G}$ ($m \in \mathbb{N}_0$). These projections are defined
by
\[
 \Pi_{m}F=\sum_{n=0}^{m} I_{n} ( F_{n}),
\]
where $F=\sum_{n=0}^{\infty} I_{n} (F_{n})\in \mathcal{G}^{\ast}$.

\begin{proposition}\label{app1} $\mathbf{a)}$ Let $f=\sum_{n=0}^{\infty} I_{n} ( f_{n})\in \mathcal{G}$. Then
\[
\lim_{m \rightarrow \infty} \Pi_{m}f = f.
\]

$\mathbf{b)}$ Let $F=\sum_{n=0}^{\infty} I_{n} ( F_{n})\in
\mathcal{G}^{\ast}$. Then
\[
\lim_{m \rightarrow \infty}\Pi_{m}F = F.
\]
\end{proposition}

\begin{proof}
$\mathbf{a)}$ Let $p>0$ and $\varepsilon>0$. Since $f \in
\mathcal{G}$, there exists $m_{0} \in \mathbb{N}_0$  such that
\[
\sum_{n=m+1}^{\infty} n! e^{2pn} \ | f_{n}|_{0}^{2}<
\varepsilon^{2}
\]
for $m \geq m_0$.

Then
\begin{eqnarray*}
\| \Pi_{m}f - f\|_{p}^{2} & = &  \sum_{n=m+1}^{\infty} n! e^{2pn}
\ | f_{n}|_{0}^{2}\\ & < & \varepsilon^{2}
\end{eqnarray*}
for $m\geq m_{0}$. We conclude that $\lim_{m \rightarrow \infty}
\Pi_{m}f = f$.

\noindent $\mathbf{b)}$ Let $\varepsilon >0$. Since  $F\in
\mathcal{G}^{\ast}$ there exists $q \in \mathbb{N}_0$ such that
\[
 \| F\|_{-q}^{2} =\sum_{n=0}^{\infty} n! e^{-2qn} \ | F_{n}|_{0}^{2} <
 \infty.
\]
Then there exists  $m_{0} \in \mathbb{N}_0$ such that
\[
 \sum_{n=m+1}^{\infty} n!  e^{-2qn} \ | F_{n}|_{0}^{2}<
 \varepsilon^{2}.
\]
for $m \geq m_0$.

Thus
\begin{eqnarray*}
\| \Pi_{m}F - F\|_{-q}^{2} & = &  \sum_{n=m+1}^{\infty} n!
e^{-2qn} \  | F_{n}|_{0}^{2} \\
& < & \varepsilon^{2},
\end{eqnarray*}
for $m\geq m_{0}$.
\end{proof}

\noindent Now, we introduce the following products of stochastic
distributions.

\begin{definition}
Let $F$ and $G$ be stochastic distributions. The products $F
\bullet_{r} G$, $F \bullet_{l} G$ and $F \bullet G$ are, by
definition,
\begin{equation}\label{1dh}
F \bullet_{r} G=\lim_{m \rightarrow \infty}F (\Pi_{m} G)
\end{equation}

\begin{equation}\label{2dh}
F\bullet_{l} G=\lim_{m \rightarrow \infty} ( \Pi_{m} F)G
\end{equation}

\begin{equation}\label{3dh}
F\bullet G=\lim_{m \rightarrow \infty} ( \Pi_{m} F) (\Pi_{m} G).
\end{equation}
\noindent where the limits are taking in the weak sense.
\end{definition}

\begin{proposition}\label{extendpro}
The products (\ref{1dh}), (\ref{2dh}) and (\ref{3dh}) extend the
product of elements of $\mathcal{G}$ and $\mathcal{G}^{\ast}$.
\end{proposition}

\begin{proof}
It follows from the bi-continuity of the product of elements of
$\mathcal{G}$ and $\mathcal{G}^{\ast}$.
\end{proof}

\noindent We recall that two elements
$F=\sum_{n=0}^{\infty}I_n(F_n)$ and
$G=\sum_{n=0}^{\infty}I_n(G_n)$ of $\mathcal{G}^{\ast}$ are
strongly independent (see F. Benth and J. Potthoff \cite{BePo}),
if there exists two intervals $I_F$ and $I_G$ whose intersection
has Lebesgue measure zero such that $supp(F_n)\subseteq I_F^n$ and
$supp(G_n)\subseteq I_G^n$ for all $n \in \mathbb{N}$.

\begin{proposition} Let $F$ and $G$ be strongly
independent in $\mathcal{G}^{\ast}$. Then
\[
F \bullet G = F \diamond G.
\]
\end{proposition}

\begin{proof} Let $F=\sum_{n=0}^{\infty} I_{n}(F_{n})$ and $G=\sum_{k=0}^{\infty}
I_{k}(G_{k})$ be strongly independent in $\mathcal{G}^{\ast}$. It
is clear that $\lim_{m \rightarrow \infty} ( \Pi_{m} F) (\Pi_{m}
G)$ is equal to
\[
\lim_{m \rightarrow \infty} \sum_{n=0}^{m} \sum_{k=0}^{m}
\sum_{j=0}^{n\wedge k} k! \ {n \choose j} {k \choose j}
I_{n+k-2j}(F_{n} \hat{\otimes}_{j} G_{k}).
\]

\noindent Since $F$ and $G$ are strongly independent in
$\mathcal{G}^{\ast}$,
\[
\sum_{j=0}^{n\wedge k} k! \ {n \choose j} {k \choose j}
I_{n+k-2j}(F_{n} \hat{\otimes}_{j} G_{k})=I_{n+k}(F_{n}
\hat{\otimes} G_{k}).
\]
Combining the above equalities we conclude that
\begin{eqnarray*}
F\bullet G & = & \lim_{m \rightarrow \infty} ( \Pi_{m} F) (\Pi_{m}
G)  \\
& = & \lim_{m \rightarrow \infty} \sum_{n=0}^{m} \sum_{k=0}^{m}
I_{n+k}(F_{n} \hat{\otimes} G_{k}) \\
& = & F \diamond G.
\end{eqnarray*}

\end{proof}

\section{Stochastic generalized functions}

\noindent In order to introduce an algebra of Colombeau type such
that include the stochastic distributions we consider
$\mathcal{G}^{\mathbb{N}_0}$ the space of sequences of
$\mathcal{G}$. It is clear that $\mathcal{G}^{\mathbb{N}_0}$ has
the structure of an associative, commutative algebra with the
natural operations:

\begin{eqnarray}\label{operaS}
(\varphi_{n})+(\psi_{n})& = & (\varphi_{n}+ \psi_{n}) \\
a(\varphi_{n})& = & (a\varphi_{n}) \\
(\varphi_{n})\cdot (\psi_{n})& = & (\varphi_{n} \psi_{n})
\end{eqnarray}
where $(\varphi_{n})$ and $(\psi_{n})$ are in $\mathcal{G}$ and $a
\in \mathbb{R}$.

\noindent We introduce the following sequence spaces. Let
$\mathbf{e}$ be given by
\[
\mathbf{e}=\{(a_n)\in \ell^2:  \lim_{n\rightarrow\infty}e^{nm}
a_n=0, \; \mbox{for all } \; m \in \mathbb{N}_0\}.
\]
We consider $\mathbf{e}$ equipped with the sequence of norms ($m
\in \mathbb{N}_0$)
\[
| (a_n)|_m = (\sum_{n=0}^{\infty}  e^{2nm}  \mid
a_n\mid^2)^{\frac{1}{2}}
\]
or with the equivalent sequence of norms
\[
\mid(a_{n})\mid_{m,\infty} =  \sup_{n} e^{nm}| a_n|.
\]
It is clear that $\mathbf{e}$ is a locally convex space.

\noindent Let us denote by $\mathbf{e}^{\prime}$ the topological
dual space to $\mathbf{e}$ (see \cite{val}). It follows that
$\mathbf{e}^{\prime}$ is given by
\[
\mathbf{e}^{\prime}=\{(b_n): \mbox{for some
}\;(C,m)\in\mathbb{R}\times \mathbb{N}_0, \; \mid b_n \mid \leq
Ce^{2nm} \mbox{ for all } n \}.
\]
\noindent The natural pairing of elements from $\mathbf{e}$ and
$\mathbf{e}^{\prime}$, denoted by $\langle \cdot, \cdot \rangle$,
is given by
\[
\langle (b_n), (a_n) \rangle = \sum_{n=0}^{\infty}b_na_n
\]
for $(b_n) \in \mathbf{e}^{\prime}$ and $(a_n) \in \mathbf{e}$.

\noindent We observe that $\mathbf{e}^{\prime}$ is an algebra with
the pointwise operations and $\mathbf{e}$ is an ideal.

\begin{definition}\label{defeal} Let
\[
G_{\mathbf{e}}=\{(\varphi_{m})\in  \mathcal{G}^{\mathbb{N}_{0}} \
: \ \forall \  p\in \mathbb{N}_{0} \ (\| \varphi_{m} \|_{p})\in
\mathbf{e} \ \},
\]
and
\[
G_{\mathbf{e}^{\prime}}=\{(\varphi_{m})\in
\mathcal{G}^{\mathbb{N}_{0}} \ : \ \forall \ p\in \mathbb{N}_{0} \
(\| \varphi_{m} \|_{p})\in \mathbf{e}^{\prime}  \ \}.
\]
\end{definition}

\begin{lemma}\label{lema} $G_{\mathbf{e}^{\prime}}$ is a subalgebra
of $\mathcal{G}^{\mathbb{N}_0}$ and $G_{\mathbf{e}}$ is an ideal of
$G_{\mathbf{e}^{\prime}}$.
\end{lemma}

\begin{proof}
Let $(\varphi_n),(\psi_n) \in G_{\mathbf{e}^{\prime}}$ and $m \in
\mathbb{N}_0$. From (\ref{eq1}) it follows that there exists $r, s
\in \mathbb{N}_0$ and a constant $C_m>0$ such that
\[
\| \varphi_n \psi_n\|_m  \leq  C_m \| \varphi_{n} \|_r\| \psi_n
\|_s.
\]
By definition, there exists constants $D, E>0$ and $p, q\in
\mathbb{N}_0$ such that
\begin{eqnarray*}
\| \varphi_n \|_r & \leq & De^{np} \\
\| \psi_n \|_s & \leq & Ee^{nq}.
\end{eqnarray*}
Combining these inequalities, we obtain
\[
\| \varphi_n \psi_n\|_m  \leq C_mDE e^{n(p+q)}.
\]
This proves that $(\| \varphi_n \psi_n\|_m ) \in
\mathbf{e}^{\prime}$, thus $(\varphi_n)\cdot(\psi_n) \in
G_{\mathbf{e}^{\prime}}$.

\noindent It remains to prove that $G_{\mathbf{e}}$ is an ideal of
$G_{\mathbf{e}^{\prime}}$. Let $(\varphi_n) \in
G_{\mathbf{e}^{\prime}}$, $(\psi_n) \in G_{\mathbf{e}}$ and $m \in
\mathbb{N}_0$. By  definitions, we have that for each $r \in
\mathbb{N}_0$ there exists a constant $D>0$ and $p\in
\mathbb{N}_0$ such that
\[
\|\varphi_n \|_r  \leq  De^{np}
\]
and for all $s,l \in \mathbb{N}_0$,
\[
|(\| \psi_n\|_s)|^2_l =\sum_{n=0}^{\infty} e^{2lp}\| \psi_n\|^2_s
< \infty.
\]

\noindent Combining the inequality (\ref{eq1}) with the above
equations we obtain
\begin{eqnarray*}
|(\| \varphi_n \psi_n\|_m)|^2_l & = & \sum_{n=0}^{\infty}
e^{2nl}\| \varphi_n
\psi_n\|^2_m \\
& \leq &  C^2_m \sum_{n=0}^{\infty}e^{2nl} \| \varphi_{n} \|^2_r\| \psi_n \|^2_s \\
 & \leq & C^2_mD^2\sum_{n=0}^{\infty} e^{2n(l+p)} \| \psi_n \|^2_s \\
 & < & \infty.
\end{eqnarray*}
We have proved that $(\| \varphi_n \psi_n\|_m) \in \mathbf{e}$,
for all $m \in \mathbb{N}_0$, this is $(\varphi_n)\cdot(\psi_n)
\in G_{\mathbf{e}}$.
\end{proof}

\begin{proposition}\label{includ} Let $ F\in \mathcal{G}^{\ast}$. Then
$(F_{m})\in  G_{\mathbf{e}^{\prime}}$, where $F_{m}=\Pi_{m}F$.
\end{proposition}

\begin{proof}
\noindent Since  $F=\sum_{n =0}^{\infty}I_n(F_n)\in
\mathcal{G}^{\ast}$, there exists $q \in \mathbb{N}_0$ such that
\[
 \| F\|_{-q}^{2} =\sum_{n=0}^{\infty} n!  \ e^{-2qn} \ | F_{n}|_{0}^{2} <
 \infty.
\]
Then
\begin{eqnarray*}
\|F_{m}\|_{p}^{2} & = & \sum_{n=0}^{m}    n!   \ e^{2pn} \ |
F_{n}|_{0}^{2} \\
& \leq & e^{(2p+2q)m} \  \| F\|_{-q}^{2}
\end{eqnarray*}
for all $p \in \mathbb{N}_0$. This proves that $(F_{m}) \in
G_{\mathbf{e}^{\prime}}$.
\end{proof}

\begin{definition}
The algebra of stochastic generalized functions is defined as
\[
\mathbf{G}=   G_{\mathbf{e}^{\prime}} /  G_{\mathbf{e}}.
\]
The elements of $\mathbf{G}$ are called stochastic generalized
functions.
\end{definition}

\noindent Let $(\varphi_m) \in G_{\mathbf{e}^{\prime}}$ we will
denote by $[\varphi_m]$ the equivalent class $(\varphi_m) +
G_{\mathbf{e}}$.

\begin{proposition}\label{includ2}
Let $\iota:\mathcal{G}^{\ast}\rightarrow \mathbf{G}$ be the
application
\[
\iota(F)=[F_m].
\]
Then $\iota$ is a linear embedding. Moreover,we have that  for all
$\varphi \in \mathcal{G}$,

\[
\iota(\varphi)=[\varphi].
\]
\end{proposition}

\begin{proof}
It is clear from the above Proposition, that $\iota$ is a well
defined linear application. Let $F \in \mathcal{G}^{\ast}$ such
that $\iota(F)=0$. As $(F_m) \in G_{\mathbf{e}}$ we have $\lim_{m
\rightarrow \infty}\|F_{m}\|_{p}=0$, for all $p \in \mathbb{N}_0$.
This is the sequence $(F_{m})$ converge weakly to $0$, it follows
that $F=0$.

\noindent It remains to prove that $ \iota(\varphi)=[\varphi]$,
for all $\varphi \in \mathcal{G}$. Let $\varphi=\sum_{n
=0}^{\infty}I_n(f_n)\in \mathcal{G}$ and $p \in \mathbb{N}_0$.
Then
\begin{eqnarray*}
e^{2am}\|\varphi-\Pi_{m}\varphi\|_{p}^{2} & = &
e^{2am}\sum_{n=m+1}^{\infty}
n! \ e^{2pn} \ | f_{n}|_{0}^{2} \\
& \leq &  \sum_{n=m+1}^{\infty} n! \ e^{2n(p+a)} \ |
f_{n}|_{0}^{2},
\end{eqnarray*}
for all $a \in \mathbb{N}_0$.

\noindent Since $\varphi \in \mathcal{G}$, we have
\[
\sum_{n=m+1}^{\infty} n! \ e^{2n(p+a)} \ | f_{n}|_{0}^{2} <
\infty.
\]
Combining these facts, we have that
\[
\lim_{m \rightarrow
\infty}e^{2am}\|\varphi-\Pi_{m}\varphi\|_{p}^{2} = 0
\]
for all $a \in \mathbb{N}_0$. This shows that
$(\varphi-\Pi_{m}\varphi) \in G_{\mathbf{e}}$, which completes the
proof.
\end{proof}

\begin{corollary} Let $\varphi,\psi\in \mathcal{G}$. Then
\[
\iota(\varphi\psi)=\iota(\varphi) \cdot \iota(\psi).
\]
\end{corollary}
\begin{proof}
We first observe that
\[
(\varphi\psi)-(\varphi_{m})\cdot(\psi_{m})=(\varphi)\cdot(\psi-\psi_{m})+(\varphi-\varphi_{m})\cdot(\psi).
\]
Applying Proposition \ref{includ2} and Lemma \ref{lema} we obtain
$(\varphi\psi)-(\varphi_{m})\cdot(\psi_{m}) \in G_{\mathbf{e}}$.
Therefore $\iota(\varphi\psi)=\iota(\varphi) \cdot \iota(\psi)$.
\end{proof}

\noindent Now, we give some examples of stochastic generalized
functions.
\begin{example} The white noise generalized function is defined by

$$
W_{x}=[W_{m}(x)]=  [I_{1} ( \sum_{j=0}^{m} \eta_{j}(x)
\eta_{j}(\cdot))].
$$

\noindent For $d=1$ we have the following identity

\begin{equation}\label{dar0}
 \sum_{j=0}^{n} \eta_j(t)
 \eta_j(x)=\frac{\sqrt{n+1}}{x-t}\Big(\eta_{n+1}(x)\eta_{n}(t)-\eta_{n+1}(t)\eta_{n}(x)\Big).
\end{equation}

Thus
\[
W_{x}= [\int
\frac{\sqrt{n+1}}{x-t}\Big(\eta_{n+1}(x)\eta_{n}(t)-\eta_{n+1}(t)\eta_{n}(x)\Big)
\ dB(t)].
\]

\end{example}

\begin{example} The element $W_{x}^{2}$. From the above example we have

\[
W_{x}^{2}=[W_{m}(x)^{2}]= [I_{1} ( \sum_{j=0}^{m} \eta_{j}(x)
\eta_{j}(\cdot)) \  I_{1} ( \sum_{j=0}^{m} \eta_{j}(x)
\eta_{j}(\cdot))].
\]

\noindent From the formula (\ref{ten2}) we get
\[
W_{x}^{2}=[W_{m}(x)\diamond W_{m}(x) + \ \sum_{j=0}^{m}
\eta_{j}^{2}(x)].
\]
\end{example}
\begin{example} Let $x\in \mathbb{R}$ and $t>0$, the
Donsker delta is the stochastic generalized function defined by
\[
\iota(\delta_{x}(B_{t}))=[\sum_{n=0}^{m} \frac{1}{n!} t^{-n}
H_{n,t}(x) H_{n,t}(B_{t})],
\]
where $H_{n,t}$ is the n-th Hermite polynomial with variance $t$
(see for instance \cite{Hida}).
\end{example}

\begin{definition}\label{ass1} Let $[\varphi_m]$ and $[\psi_m]$ be stochastic generalized functions. We say
that $[\varphi_m]$ and $[\psi_m]$ are associated, denoted by $
[\varphi_m] \approx [\psi_m]$, if for all $\varphi \in
\mathcal{G}$ we have that
\[
\lim_{m\rightarrow\infty}\ll \varphi_m-\psi_m,\varphi\gg=0.
\]
\end{definition}

\noindent We observe that the relation $\approx$ is well defined
as an equivalence relation on $\mathbf{G}$.

\begin{lemma}\label{equias}
Let $F$ and $G$ be in $\mathcal{G}^{\ast}$. Then:
\begin{enumerate}
\item If $\iota(F)\approx \iota(G)$ then $F=G$.

\item If there exists $F \bullet G$ then $\iota(F) \cdot \iota(G)
\approx \iota(F \bullet G)$.

\item If $\psi\in \mathcal{G}$ and $F\in \mathcal{G}^{\ast}$ then
$\iota(\psi)\iota(F)\approx \psi F$.
\end{enumerate}
\end{lemma}

\begin{proof}

$(1)$ Combining Proposition \ref{app1} and $\iota(F)\approx
\iota(G)$ we conclude that,
\[
\ll F-G,\varphi\gg = \lim_{m\rightarrow\infty} \ll (\Pi_mF-\Pi_mG)
\ \varphi \gg=0
\]
for all $\varphi\in \mathcal{G}$. Thus $F=G$.

\noindent The proofs of $(2)$ and $(3)$ are straightforward.
\end{proof}

\begin{remark}
We observe that our construction of stochastic generalized
functions also can be done for others space of test functions and
distributions spaces, for instance Hida and Kondratiev spaces.
\end{remark}

\noindent We introduce the generalized sequences ring $\mathbf{g}$
with the purpose of to define the expectation of generalized
functions of $\mathbf{G}$.

\begin{definition}\label{genseq} The generalized sequences ring,
denoted by $\mathbf{g}$, is define to be
\begin{equation}
\mathbf{g}=\mathbf{e}^{\prime}/ \mathbf{e}.
\end{equation}

\noindent Let $(b_n) \in \mathbf{e}^{\prime}$ we will use $[b_n]$
by denoted the equivalent class $(b_n) + \mathbf{e}$.
\end{definition}
It is clear from the definition that $\mathbf{e}^{\prime}$ is a ring
but is not a field. In fact, $\mathbf{e}^{\prime}$ has zero
divisors.
\begin{lemma} $\mathbf{a})$ Let $\iota_0:\mathbb{R}\rightarrow \mathbf{g}$ be the application
\[
\iota_0(a)=[a].
\]
Then $\iota_0$ is an embedding.

$\mathbf{b})$ $\mathbf{G}$ is a $\mathbf{g}$-module with the
natural operations.
\end{lemma}

\begin{proof}
The proof is straightforward.
\end{proof}
\noindent We define the notion of association in $\mathbf{g}$.
\begin{definition}
Let $[a_n]$ and $[b_n]$ be in $\mathbf{g}$. We says that $[a_n]$
and $[b_n]$ are associated, denoted by $ [a_n] \approx [b_n]$, if
\[
\lim_{n\rightarrow\infty}(a_n-b_n)=0.
\]
\end{definition}
\noindent We observe that the relation $\approx$ is well defined
equivalence relation on $\mathbf{g}$.

\begin{proposition} Let $[\varphi_m]=[\psi_{m}] \in \mathbf{G}$. Then
$[\mathbb{E}(\varphi_m)]=[\mathbb{E}(\psi_m)] \in \mathbf{g}$.
\end{proposition}

\begin{proof}
Let $[\varphi_{m}] \in G$. By definition, $(\varphi_m)\in
G_{\mathbf{e}^{\prime}}$. Thus for all $p \in \mathbb{N}_0$ we
have that $(\|\varphi_m\|_p )\in \mathbf{e}^{\prime}$. It follows
that for all $p \in \mathbb{N}_0$ there exists given $C>0$ and $a
\in \mathbb{N}_0$ such that
\[
\|\varphi_m\|^2_p \leq C^2e^{2am}
\]
for all $m \in \mathbb{N}_0$.

\noindent By definition of $\| \cdot \|_p$, we have that $|
\mathbb{E}(\varphi_m)|^{2} \leq \|\varphi_m\|^2_p$. Thus
\[
| \mathbb{E}(\varphi_m)| \leq Ce^{am}
\]
for all $m \in \mathbb{N}_0$. This implies that
$[\mathbb{E}(\varphi_m)] \in \mathbf{g}$.

It remains to prove the independence of the representative. By
definition, $(\varphi_m-\psi_m)\in G_{\mathbf{e}}$. Thus for all
$p \in \mathbb{N}_0$ we have that $(\|\varphi_m-\psi_m\|_p )\in
\mathbf{e}$. It follows that given $a, p \in \mathbb{N}_0$ and
$\varepsilon
>0$ there exists $m_0 \in \mathbb{N}_0$ such that
\[
e^{2am}\|\varphi_m-\psi_m\|^2_p < \varepsilon^2
\]
for all $m \geq m_0$.

\noindent By the definition of $\| \cdot \|_p$, we have that $|
\mathbb{E}(\varphi_m)-\mathbb{E}(\psi_m)|^{2} \leq
\|\varphi_m-\psi_m\|^2_p$.

\noindent Thus
\[
e^{am}| \mathbb{E}(\varphi_m)-\mathbb{E}(\psi_m)| < \varepsilon
\]
for all $m \geq m_0$. We conclude that
$[\mathbb{E}(\varphi_m)]=[\mathbb{E}(\psi_m)]$.
\end{proof}

\begin{definition}\label{exp}
Let $F=[\varphi_{m}]$ be a generalized function in $\mathbf{G}$.
The expectation of $F$, denoted by $\mathbb{E}(F)$, is defined to
be the generalized number
\[
\mathbb{E}(F)=[\mathbb{E}(\varphi_{m}].
\]
\end{definition}

\begin{proposition} Let $f \in L^{1}(\mu)$ and $F=[\varphi_m]\in \mathbf{G}$ such that $F\approx
f$. Then
\[
\iota_0( \mathbb{E}(f))=\mathbb{E}(F).
\]
\end{proposition}
\begin{proof}
The proof is straightforward from the definitions.
\end{proof}

\section{A Stochastic Cauchy Problem}

\noindent In this section we study the stochastic Cauchy problem
\begin{equation}\label{para}
 \left \{
\begin{array}{lll}
u_t & = & Lu  + u   W(t,x) \\
u_0 & = & f
\end{array}
\right .
\end{equation}
where $L$ is an uniformly elliptic partial differential operator,
$f\in C_{b}^{2}(\mathbb{R}^{l})$ and $W$ is the white noise
parametric generalized function defined below in Example
\ref{whitenoise}.

\noindent In order to solve (\ref{para}) we introduce the algebra
$G^{\alpha}(D)$ of parametric generalized functions. We proceed in
a similar way to the construction of the algebra $G$, the details
are left to the reader.

\noindent Let $D$ a domain in $\mathbb{R}^{d}$ and $\alpha\in
\mathbb{N}_{0}^{d} $. We denote by $C^{\alpha}(D,\mathcal{G})$ the
set of functions $f:D \times \mathcal{S}^{\prime}(\mathbb{R}^d)
\rightarrow \mathbb{R}~(\mathbb{C})$ such that $f(x,\cdot)\in
\mathcal{G}$ for all $x \in D$ and $f(\cdot, \omega) \in
C_b^{\alpha}(D)$ for almost $\omega \in
\mathcal{S}^{\prime}(\mathbb{R}^d)$.

\begin{definition}
Let $ G_{\mathbf{e}^{\prime}} ^{\alpha}(D)$ be the set of $(f_m)
\in (C^{\alpha}(D,\mathcal{G}))^{\mathbb{N}_0}$ such that
\[
(\sup_{x}\| D^{\beta} f_{m}(x,\cdot) \|_{p})\in
\mathbf{e}^{\prime}
\]
for all $\beta\in \mathbb{N}_{0}^{d}$ with $\beta \leq \alpha$ and
$p\in \mathbb{N}_{0}$.
\end{definition}
\noindent We observe that $ G_{\mathbf{e}^{\prime}} ^{\alpha}(D)$
is a subalgebra of $(C^{\alpha}(D,\mathcal{G}))^{\mathbb{N}_0}$
and for $\gamma \in \mathbb{N}^d_0$, $0 \leq \gamma \leq \alpha$
the partial differential operator $D^{\gamma}$ transform
$G_{\mathbf{e}^{\prime}} ^{\alpha}(D)$ into $
G_{\mathbf{e}^{\prime}} ^{\alpha-\gamma}(D)$.

\begin{definition}\label{defeald} Let $
G_{\mathbf{e}}^{\alpha}(D)$ be the set of $(f_m) \in
(C^{\alpha}(D,\mathcal{G}))^{\mathbb{N}_0}$ such that
\[
(\sup_{x}\| D^{\beta} f_{m}(x,\cdot) \|_{p})\in \mathbf{e}
\]
for all $\beta\in \mathbb{N}_{0}^{d}$ with $\beta \leq \alpha$ and
$p\in \mathbb{N}_{0}$.
\end{definition}
\noindent It is clear that $G_{\mathbf{e}}^{\alpha}(D)$ is an
ideal of $G_{\mathbf{e}^{\prime}}^{\alpha}(D)$.

\begin{definition}
We define the algebra of parametric generalized functions
$G^{\alpha}(D)$ as
\[
G_{\mathbf{e}^{\prime}}^{\alpha}(D)
    / G_{\mathbf{e}}^{\alpha}(D).
\]
The elements of $G^{\alpha}(D)$ are called parametric generalized
functions.
\end{definition}
\noindent In the case that $\alpha=(0,..,0)$ we denote
$G^{\alpha}(D)$ by $G(D)$. It is clear that $\mathbf{G}\subset
G^{\alpha}(D)$ and $C_b^{\alpha}(D) \subset G^{\alpha}(D)$.

\begin{example}\label{whitenoise}
Let $D$ be a domain in $\mathbb{R}^{d}$ and $\alpha \in
\mathbb{N}_0^d$. The white noise parametric generalized function
is defined by
$$
W_{x}=(W_{m}(x))=  (I_{1} ( \sum_{j=0}^{m} \eta_{j}(x)
\eta_{j}(\cdot))),
$$
for $x \in D$.

\noindent In fact, we observe that there exists constants
$C_{\alpha}$ and $a_{\alpha}$ such that for all $0 \leq \beta \leq
\alpha$,

\begin{eqnarray*}
\| D_x^{\beta} W_{m}(x) \|_{p}^{2} & = & \| D_x^{\beta} W_{m}(x)
\|_0^{2} \\
& = &  \| I_{1} (\sum_{j=0}^{m} D_x^{\beta}\eta_{j}(x)
\eta_{j}(\cdot))
\|_0^2 \\
& = & |\sum_{j=0}^{m} D_x^{\beta}\eta_{j}(x) \eta_{j}(\cdot)|_0^2
\\ & = & \sum_{j=0}^{m} |D_x^{\beta}\eta_{j}(x)|^2 \\
& \leq & C_{\alpha}(m+1)^{a_{\alpha}}.
\end{eqnarray*}
\end{example}

\begin{definition}
We says that $u=[u_m] \in
G^{(1,2,...,2)}([0,T]\times\mathbb{R}^{l})$ is a {\it generalized
solution} of the Cauchy problem (\ref{para}) if $u_t = Lu  + u
W(t,x)$ in $G([0,T]\times\mathbb{R}^{l})$ and $u_0 = f$ in
$G^{(2,..,2)}(\mathbb{R}^{l})$.
\end{definition}
\noindent We observe that $[u_m] \in G^{\alpha}(D)$ implies that
$[u_m(0,\cdot)]$ is a well defined element of
$G^{(2,..,2)}(\mathbb{R}^{l})$.

\begin{theorem}\label{teoStem}
There exists an unique generalized solution $u=[u_m]$ in
$G^{\alpha}(D)$ for the Cauchy problem (\ref{para}).
\end{theorem}
\begin{proof}
Let $d=1+l$, $\alpha=(1,2,...,2)\in \mathbb{N}_0^{1+l}$ and
$D=[0,T]\times\mathbb{R}^{l}$ ($T>0$). We consider that $L$ have
the form
\[
Lf= \frac{1}{2} \sum_{i,j=1}^{l} a_{ij} \frac{\partial^2}{\partial
x_i \partial x_j}f+ \sum_{i=1}^{l} b_{i} \frac{\partial}{\partial
x_i} f
\]
where $a_{ij}$ and $b_{j}$ belong to $C_{b}^{2}(\mathbb{R}^{l})$
for all $i$ and $j$.

\noindent We split the proof in three steps:

\noindent (a) We solve the family of parabolic problems ($m \in
\mathbb{N}_0$)
\begin{equation}\label{netparabo}
\frac{d}{dt}u_{m}=  Lu_{m}  + u_{m}   W_{m}(t,x)
\end{equation}
with $u(x,0)=f(x)$.

\noindent(b) We check that the nets of solutions $(u_{m})$ belongs
to $G_{\mathbf{e}^{\prime}}^{\alpha}(D)$.

\noindent(c) We check that if $(u_{m})$ and $(v_{m})$ are two nets
of solutions of (\ref{para}), then $[u_{m}]=[v_{m}]$.

\noindent In order to prove (a), we use that the Feymann-Kac
formula give a representation of the solution of (\ref{netparabo})
(see for instance \cite{Frei}). We observe that there exist
$\sigma=(\sigma_{ij}) \in C^2(\mathbb{R}^l;\mathbb{R}^{l \times
l})$ such that
\[
a_{ij}=\sum_{k=1}^{l}  \sigma_{ik} \sigma_{jk}.
\]
We consider the following stochastic differential equation,
\begin{equation}\label{itoass}
dX=\sigma(X) \ dB + b(X) \ dt,
\end{equation}
where $B$ is an $l$-dimensional Brownian motion in an auxiliar
probability space, $\sigma=(\sigma_{ij})$ and $b=(b_i)$.

\noindent The solution of (\ref{itoass}) with $X(0)=x\in
\mathbb{R}^{l}$ is denoted by $X(t,x)$. Applying the Feynman-Kac
formula to $u_{m}$ we have
\begin{equation}\label{um0}
u_{m}(t,x)=\mathbb{E}(f(X(t,x)) \ e^{\int_{0}^{t}
W_{m}(t-s,X(s,x)) \ ds } ),
\end{equation}
where $\mathbb{E}$ denotes the expectation in the auxiliar
probability space.

\noindent It remains to prove b), this is $(u_{m})$ belongs to
$G_{\mathbf{e}^{\prime}}^{\alpha}(D)$.

\noindent We claim that for each $p \in \mathbb{N}_0$ there exists
constants $C$ and $a$ such that
\[
\sup_{(t,x)}\|  u_{m}(t,x) \|_{p},~\sup_{(t,x)}\|
\frac{d}{dt}u_{m}(t,x) \|_{p}, ~ \sup_{(t,x)}\|
D^{\beta}u_{m}(t,x) \|_{p} \leq C e^{am}
\]
where $0 \leq \beta \leq (2,...,2)$.

\noindent In fact, by the definition of Wick exponential
(\ref{exwick}),
\begin{equation}\label{um}
 u_{m}(t,x)= \mathbb{E}(f(X(t,x)) \ e^{\frac{1}{2}
|\int_{0}^{t} \sum_{j=0}^{m} \eta_{j}((t-s),X(s,x))
\eta_{j}(\cdot) \ ds|_{0}^{2}} \ e^{: \int_{0}^{t}
W_{m}(t-s,X(s,x)) \ ds } ).
\end{equation}
\noindent From the uniform boundedness and orthonormality of the
Hermite functions, we have that there exist a positive constant
$a$ depending of $T$ such that
\begin{equation}\label{acot wick}
|\int_{0}^{t} \sum_{j=0}^{m} \eta_{j}((t-s),X(s,x))
\eta_{j}(\cdot)ds|^2_0 \leq am.
\end{equation}
\noindent Combining (\ref{um}), (\ref{acot wick}) and
(\ref{exwick1}), we obtain that $\sup_{(t,x)}\|  u_{m}(t,x)
\|_{p}^{2}$ is bounded by
\[
(\sup_{x \in \mathbb{R}^l}|f(x)|)^{2}e^{am}.
\]
\noindent In order to show that  $\sup_{(t,x)}\|
\frac{d}{dt}u_{m}(t,x) \|_{p} \leq C e^{am}$, applying the It\^o
formula in (\ref{um0}), we obtain that
\begin{eqnarray*}
\frac{d}{dt}u_{m}(t,x)& = & \mathbb{E}(\{f(X(t,x))(
W_{m}(0,X(t,x)) + \ \int_{0}^{t}  \frac{d}{dt}W_{m}(t-s,X(s,x))\ ds) \\
& & + \ Lf_{m}(X(t,x))\}  \ e^{\int_{0}^{t} W_{m}(t-s,X(s,x)) \ ds
} ).
\end{eqnarray*}
\noindent We need dominate the following terms:
\begin{enumerate}
\item[a)] $A_1=\|\mathbb{E}(f(X(t,x))W_{m}(0,X(t,x))
e^{\int_{0}^{t} W_{m}(t-s,X(s,x)) \ ds } \|_p^2$,

\item[b)] $A_2=\|\mathbb{E}(f(X(t,x))\int_{0}^{t}
\frac{d}{dt}W_{m}(t-s,X(s,x))\ ds~ e^{\int_{0}^{t}
W_{m}(t-s,X(s,x)) \ ds } \|_p^2$,

\item[c)] $A_3=\|\mathbb{E}(Lf(X(t,x)) e^{\int_{0}^{t}
W_{m}(t-s,X(s,x)) \ ds } \|_p^2$.
\end{enumerate}
\noindent $\mathrm{a)}$ We observe that $A_1$ is equal to
\begin{eqnarray*}
\|\mathbb{E}(f(X(t,x))e^{\frac{1}{2} |\int_{0}^{t} \sum_{j=0}^{m}
\eta_{j}((t-s),X(s,x)) \eta_{j}(\cdot) \
ds|_{0}^{2}}W_{m}(0,X(t,x)) e^{: \int_{0}^{t} W_{m}(t-s,X(s,x)) \
ds } \|_p^2.
\end{eqnarray*}
By the boundeness of $f$ and the inequality (\ref{acot wick}) we
have that
\[
A_1 \leq Ce^{am}\mathbb{E}(\|  W_{m}(0,X(t,x)) e^{: \int_{0}^{t}
W_{m}(t-s,X(s,x)) \ ds } \|_p^2).
\]
From (\ref{eq1}) there exists $C_p>0$ and $r,r^{\prime}\in
\mathbb{N}_0$ such that
\[
\|  W_{m}(0,X(t,x)) e^{: \int_{0}^{t} W_{m}(t-s,X(s,x)) \ ds }
\|_p^2
\]
is lower or equal to
\[
C_p \|  W_{m}(0,X(t,x))\|_r^2 \| e^{: \int_{0}^{t}
W_{m}(t-s,X(s,x)) \ ds } \|_{r^{\prime}}^2 \leq Ce^{am}.
\]
\noindent Thus $A_1 \leq Ce^{am}$.

\noindent $\mathrm{b)}$ and $\mathrm{c)}$. The estimative for
$A_2$ is obtained in a similar way and the estimative for $A_3$ is
analogous to the estimative for $\|u_m\|_p^2$. We conclude that
\[
\| \frac{d}{dt}u_{m}(t,x)\|_p \leq Ce^{am}.
\]

\noindent It remains to prove that $\sup_{(t,x) \in K }\|
D^{\beta}u_{m}(t,x) \|_{p} \leq C e^{am}$ for $0 \leq \beta \leq
(2,...,2)$.

\noindent We first observe that $D^{i}u_{m}$ ($i=1,..,l$) is equal
to

$$
\mathbb{E}( \{ \sum_{j=1}^{l} D^{i}f(X(t,x))D^{i}X_{j}(t,x) +
f(X(t,x))  \sum_{j=1}^{l} \int_{0}^{t} D^{i}W_{m}(t-s,X(t,x))
$$
$$
 \times D^{i}X_{j}(s,x)) \ ds\}e^{\int_{0}^{t}
W_{m}(t-s,X(s,x)) \ ds} ).
$$

\noindent It is clear that $D_{i}X_{j}$ satisfy the following
system of linear stochastic equations
\[
D^{i}X_{j}=   \delta_{ij} + \int_{0}^{t}  \sum_{k,h=1}^{l}
(D^{k}\sigma_{jh})(X(s,x)) D^{i}X_{k}(s,x) \ dB_{h}(s) +
\]

\begin{equation}\label{sist}
 + \int_{0}^{t}    \sum_{k=1}^{l}
(D^{k}b_{j})(X(s,x))      D^{i}X_{k}(s,x) \ ds.
\end{equation}

\noindent  From the equation (\ref{sist}) and  Gronwall lemma we get
that
\[
\mathbb{E}(\sum_{i=1}^{l}|D^{i}X_{j}(t,x)|^{2})
\]
is locally uniformly bounded in $t$ and $x$.

\noindent We need dominate the terms
\begin{enumerate}
\item[d)] $B_1= \|\mathbb{E}(  \sum_{j=1}^{l}
D^{i}f(X(t,x))D^{i}X_{j}(t,x)e^{\int_{0}^{t} W_{m}(t-s,X(s,x)) \
ds} )\|^2_p$,

\item[e)] $B_2= \|\mathbb{E}(f(X(t,x)) (\sum_{j=1}^{l}
\int_{0}^{t} D^{i}W_{m}(t-s,X(t,x))) D^{i}X_{j}(s,x)) \ ds)
\\ e^{\int_{0}^{t} W_{m}(t-s,X(s,x)) \ ds}  \|_p^2$.

\end{enumerate}
\noindent $\mathrm{d)}$ We observe that $B_1$ is equal to
\[
\|\mathbb{E}(  \sum_{j=1}^{l}
D^{i}f(X(t,x))D^{i}X_{j}(t,x)e^{:\int_{0}^{t} W_{m}(t-s,X(s,x)) \
ds} e^{\frac{1}{2} |\int_{0}^{t} \sum_{j=0}^{m}
\eta_{j}((t-s),X(s,x)) \eta_{j}(\cdot) \ ds|_{0}^{2}} ) \|^2_p.
\]
By the boundeness of the derivatives of $f$ and the inequality
(\ref{acot wick}) we have that
\[
B_2 \leq Ce^{am}
\mathbb{E}(\sum_{i=1}^l|D^iX_j(t,x)|^2\|e^{:\int_{0}^{t}
W_{m}(t-s,X(s,x)) \ ds} \|_p^2) \leq Ce^{am}.
\]

\noindent $\mathrm{e)}$ The estimative for $B_2$ is obtained in an
analogous way, using (\ref{eq1}). Therefore,
\[
\sup_{(t,x)}\| D^{i}u_{m}(t,x) \|_{p} \leq C e^{am}.
\]

\noindent The estimative for $\| D^{i}D^{j}u_m\|_p$ is done in a
similar way. So we conclude that $[u_{m}]\in G^{\alpha}(D)$.

\noindent We considerer the uniqueness. Suppose that $[u_{m}]$  and
 $[v_{m}]$ are two  generalized solutions of (\ref{para}), we denote by $d_{m}$ to $u_{m}-v_{m}$.
 By the definition, we have that $d_{m}$ satisfy
\begin{equation}\label{parauni}
 \left \{
\begin{array}{lll}
\frac{d}{dt}d_{m} & = &  Ld_{m}  + d_{m}   W_{m}(t,x) + N_{m}(t,x)\\
u_{m}(0,x) & = & N_{0,m}(x)
\end{array}
\right .
\end{equation}
\noindent with $N_{m}(t,x)\in G_{\mathbf{e}}^{\alpha}(D)$ and
$N_{0,m}\in G_{\mathbf{e}}^{\alpha}(\mathbb{R}^{l})$. Applying the
Feynman-Kac formula to $d_{m}$ we obtain that
\begin{eqnarray*}
d_{m}(t,x) & = & \mathbb{E}(N_{0,m}(X(t,x)) \ e^{\int_{0}^{t}
W_{m}(t-s,X(s,x)) \ ds } \\
& & + \int_{0}^{t}    N_{m}(s,x) \
e^{\int_{0}^{s} W_{m}(s-u,X(u,x)) \ du } \ ds ).
\end{eqnarray*}

\noindent  It is straightforward using similar techniques that
\[
\sup_{(t,x)}\|  d_{m}(t,x) \|_{p},~\sup_{(t,x)}\|
\frac{d}{dt}d_{m}(t,x) \|_{p}, ~ \sup_{(t,x)}\|
D^{\beta}d_{m}(t,x) \|_{p} \in \mathbf{e}
\]
for each $p \in \mathbb{N}_0$ and $0 \leq \beta \leq (2,...,2)$.
We conclude that the solution of (\ref{para}) is unique.

\end{proof}

\noindent We end the paper with the following remarks.

\begin{remark}
Wick solution versus generalized solution: Let us consider the
following Cauchy problem,
\begin{equation}\label{netparaboW}
 \left \{
\begin{array}{lll}
v_t & = & Lv  + v \diamond W(t,x) \\
v_0 & = & f
\end{array}
\right .
\end{equation}
where $L$ is an uniformly elliptic partial differential operator,
$f\in C_{b}^{2}(\mathbb{R}^{l})$ and $W$ is the white noise
parametric generalized function.

\noindent We observe that the proof of existence and uniqueness
given in \cite{PVW} extends to the following Cauchy problem,
\begin{equation}\label{netparaboWm}
 \left \{
\begin{array}{lll}
\frac{d}{dt}v_{m}& = &  Lv_{m}  + v_{m} \diamond   W_{m}(t,x) \\
v_m(\cdot,0) & = & f
\end{array}
\right .
\end{equation}
where $W(t,x)=[W_{m}(t,x)]$ is the white noise parametric
generalized function.

\noindent The sequence of solutions $(v_m)$ converges in the Hida
distribution space to the solution of equation (\ref{netparaboW})
(see \cite{liouz}).

\noindent We have that the S-transformation of the solution of
equation (\ref{netparaboWm}) is
\[
S(v_{m}(t,x))(h)=\mathbb{E}(f(X(t,x)) \ e^{\int_{0}^{t}
h_{m}(t-s,X(s,x)) \ ds } )
\]
and the S-transformation of the generalized solution $u=[u_m]$ of
(\ref{para}) is
\[
S(u_{m}(t,x))(h)=\mathbb{E}(f(X(t,x)) \ e^{\int_{0}^{t}
h_{m}(t-s,X(s,x)) \ ds } ) e^{| \int_{0}^{t} W_{m}( . \ ,
t-s,X(s,x)) \ ds |_{0}}.
\]
This implies that the generalized solution $u=[u_m]$ is not
associated to any stochastic distribution for $d> 1$.
\end{remark}
\begin{remark}
We can show in a similar way that there exists an unique
generalized solution of the stochastic Cauchy problem (\ref{para})
with initial data $f\in G^{(2,..,2)}(\mathbb{R}^{l})$.
\end{remark}

\end{document}